\documentclass{article}
\usepackage{graphicx}
\usepackage{subfigure}
\usepackage[T1]{fontenc}
\usepackage[latin9]{inputenc}
\usepackage{babel}
\usepackage{amsmath}
\usepackage{amsthm}
\usepackage{amssymb}
\usepackage{amsfonts}
\usepackage{authblk}
\usepackage{color}
\usepackage{enumerate}
\usepackage{natbib}
\setcitestyle{numbers,square}
\usepackage{fancyhdr}
\usepackage{float}
\usepackage{url}
\usepackage{natbib}

\usepackage{graphicx}
\usepackage{subfigure}
\usepackage[T1]{fontenc}
\usepackage[latin9]{inputenc}
\usepackage{babel}
\usepackage{amsmath}
\usepackage{amsthm}
\usepackage{amssymb}
\usepackage{amscd}
\usepackage{amsfonts}
\usepackage{enumerate}

\usepackage{natbib}
\usepackage[symbol]{footmisc}

\usepackage{latexsym}
\usepackage{verbatim}
\usepackage{bbm}
\usepackage{cases}

\newcommand{\N}{\mathbb{N}}

\newcommand{\R}{\mathbb{R}}
\newcommand{\C}{\mathbb{C}}
\renewcommand{\H}{\mathbb{H}}
\newcommand{\PSH}{\mathrm{PSH}}
\newcommand{\tr}{\mathrm{tr}}
\providecommand{\keywords}[1]
{
	\textbf{\text{Keywords: }} #1
}

\newcommand{\ddbar}{\mathfrak{i} \partial \overline{\partial}}

\theoremstyle{plain}

\newtheorem{theorem}{Theorem}[section]
\newtheorem{proposition}[theorem]{Proposition}
\newtheorem{lemma}[theorem]{Lemma}
\newtheorem{question}[theorem]{Question}

\theoremstyle{definition}

\theoremstyle{definition}
\newtheorem{definition}[theorem]{Definition}

\numberwithin{equation}{section}

\begin{document}
\begin{sloppypar}	
\title{ A uniform estimate for the quaternionic Gauduchon metric with prescribed volume form}

\author{Jiaogen Zhang
}
\affil[]{School of Mathematics,  Hefei University of Technology, \\ Hefei, 230009, PR China.}
	
\affil[]{e-mail: zjgmath@ustc.edu.cn}

\date{\today}

\maketitle
\begin{abstract}
The quaternionic Calabi conjecture, posed by Alesker and Verbitsky \cite{Alesker-Verbitsky (2010)}, predicts that the quaternionic Monge-Amp\`ere equation can always be solved on any compact HKT manifold. Motivated by this conjecture, we will introduce a quaternionic version of the Gauduchon conjecture on any compact $SL(n,\mathbb{H})$-manifold, specifically addressing the existence of quaternionic Gauduchon metrics with prescribed volume form. We reframe this question as a special case of fully nonlinear elliptic equations of second order and subsequently establish a uniform
estimate for the potential function.  
\end{abstract}

\keywords{Uniform estimates, Quaternionic Gauduchon metric, Auxiliary Monge-Amp\`ere equation}	


\section{Introduction}
A hypercomplex manifold is a smooth manifold $M$ of quaternionic dimension $ n $ equipped with three  anti-commuting complex structures $(I,J,K)$ that satisfy the quaternionic relations
\[
IJ=K\,.
\]
A Riemannian metric $ g $ on a hypercomplex manifold $ (M,I,J,K) $ is said to be {hyperhermitian} if it is Hermitian with respect to each of $ I,J $ and $K$. Any hyperhermitian metric induces a sympletic form 
$$
\Omega_{0}=\omega_J+\sqrt{-1}\omega_K
$$ 
associated to $g$, where $\omega_J$ and $\omega_K$ are the fundamental forms of $(g,J)$ and $(g,K)$ respectively. The form $\Omega_{0}$, which is of type $(2,0)$ with respect to $I$, satisfies the { $q$-real} condition $\overline{J\Omega_{0}}={\Omega}_{0} $ (here $ J $ acts on $ \Omega_{0}$ as $J\Omega_{0}(\cdot,\cdot)=\Omega_{0}(J\cdot,J\cdot)$) and is { positive} in the sense that $\Omega_{0}(X,JX)>0$ for every nowhere vanishing real vector field $X$ on $M$. A hyperhermitian manifold $(M,I,J,K,g)$ is hyperk\"ahler if and only if $\Omega_{0}$ is closed, whereas it is called HKT ({hyperk\"ahler with torsion}) if $\Omega_{0}$ is $\partial$-closed, where $\partial$ is taken with respect to $I$. The geometry of HKT manifolds is extensively studied in the literature \cite{Alesker-Verbitsky (2006),Banos,GF,GLV17,Grantcharov-Poon (2000),Howe-Papadopoulos (1996),Ivanov,Lejmi-Weber,Swann,Verbitsky (2002),Verbitsky (2007),Verbitsky (2009)} and the reference therein.

Motivated by the study of  \, \lq\lq canonical\rq\rq \, HKT metrics, in analogy to the famous Calabi conjecture demonstrated by Yau in \cite{yau}, Alesker and Verbitsky proposed in \cite{Alesker-Verbitsky (2010)} to study the {quaternionic Monge-Amp\`ere equation}:
\begin{equation}\label{qMA}
(\Omega_0+\partial\partial_J\varphi)^n={\rm e}^{F+b}\Omega^n_0
\end{equation}
on a compact HKT manifold, where $\partial_J=J^{-1}\bar \partial J$
is the twisted differential operator introduced by Verbitsky \cite{Verbitsky (2002)} and $F\in C^{\infty}(M,\R)$ is given, while $(\varphi,b)\in C^{\infty}(M,\R)\times \R$ is the unknown. Even though the solvability of the quaternionic Monge-Amp\`ere equation in its general form has remained an open problem over the past decade, several partial results have been reported in the literature \cite{Alesker (2013),Alesker-Shelukhin (2013),Alesker-Shelukhin (2017),Alesker-Verbitsky (2010),DinewSroka,GentiliVezzoni,GV,Sroka,S23}.

The solvability of Equation \eqref{qMA} has a nice geometric application, implying the existence of a unique balanced metric $\tilde g$ on a compact HKT manifold $(M,I,J,K,g)$ with holomorphically trivial canonical bundle relative to $I$. Furthermore, the form $\tilde \Omega$ induced by $\tilde g$ is a $\partial \partial_J \varphi$ perturbation of $\Omega$ (cf. \cite{Verbitsky (2009)}).  From this perspective, Equation \eqref{qMA} can be regarded as  the ``quaternionic counterpart'' of the complex Monge-Amp\`ere equation in K\"ahler geometry, where balanced HKT metrics play a role analogous to that of Calabi-Yau metrics in K\"ahler geometry.

It is well-known that a Hermitian metric $\omega$ over a compact complex manifold $X$ is referred to as \lq\lq Gauduchon metric\rq\rq \, if $$\ddbar (\omega^{\dim_{\mathbb{C}} X-1})=0\,.$$ A celebrated result by Gauduchon \cite{Gau} suggests that every Hermitian metric is conformal equivalent to a Gauduchon metric, which is unique within its conformal class except for a constant multiplier. Thus, the Gauduchon metric holds a significant position among Hermitian metrics in the study of non-K\"ahler geometry. Given the similarity to the Gauduchon metric in the complex case, exploring the quaternionic Gauduchon metric within the hyperhermitian framework appears to be a natural direction of study. Recently, Grantcharov, Lejmi, and Verbitsky \cite{GLV17} presented a quaternionic analogue of the Gauduchon metric on compact hyperhermitian manifolds, among many other intriguing ideas and findings.

\begin{definition}
A quaternionic positive $(2,0)$-form $\Omega$ on a compact hypercomplex manifold $(M,I,J,K)$ of quaternionic dimension $n$ is quaternionic Gauduchon  if $\partial\partial_J (\Omega^{n-1})$ vanishes identity.    
\end{definition}

Yau's proof of  Calabi conjecture establishes the existence of a unique K\"ahler-Einstein metric on any compact K\"ahler manifold $(M,\omega_{0})$ with its first Chern class that is either negative or zero. The latter case was subsequently  transformed into finding a new K\"ahler metric $\omega\in [\omega_{0}]$ such that its $n$-power $\omega^n$ is exactly a prescribed smooth positive volume form with a mass equal to $\int_{M}\omega_{0}^n$.  When $M$ no longer admits a K\"ahler metric, in the same spirit as the Calabi-Yau theorem, Gauduchon conjectured in 1984 that there always exists a Gauduchon metric with a prescribed volume form. 
\begin{question}\label{Gau}\cite{Gau} 
 Let $M$  be a compact complex manifold of complex dimension $n$ and $\sigma$ be a given smooth volume form on $M$. Then, there exists a Gauduchon metric $\omega_{G}$ such that  
 \[
 \omega_{G}^n=\sigma.
 \]
\end{question}

By adapting the Hodge theory, Popovici \cite{Popovici} and Tosatti-Weinkove \cite{TW19} observed a close connection between the aforementioned question and a specific form-type Monge-Amp\`ere equation studied by Fu-Wang-Wu \cite{FWW10,FWW15}, as well as the notion of $(n-1)$-PSH (plurisubharmonic in short) functions in the sense of Harvey-Lawson \cite{HL11,HL12}. Ultimately, Question \ref{Gau} was settled by Cherrier \cite{che87} for dimension two, and later by Sz\'ekelyhidi, Tosatti, and Weinkove \cite{STW17} for arbitrary dimensions.

As demonstrated by Obata \cite{Obata (1956)}, every hypercomplex manifold $(M,I,J,K)$ with quaternionic dimension $n$ admits a unique torsion free connection $\nabla^{Ob}$ on $TM$ that satisfies 
$$\nabla^{Ob} I=\nabla^{Ob} J=\nabla^{Ob} K=0\,.$$ 
Such a connection, whose holonomy group satisfies $\text{Hol}(M)\subseteq GL(n,\mathbb{H})$ with $\mathbb{H}$ represents the usual quaternions, is now referred to as Obata connection. An important subgroup of $GL(n,\mathbb{H})$ is its commutator, denoted as $SL(n,\mathbb{H})$. It may, however, reduce to  $SL(n,\mathbb{H})$ which appears in the Merkulov-Schwachh\"ofer \cite{MS99} list. In this case, 
the  hyperhermitian manifold is known as an $SL(n,\mathbb{H})$-manifold. Verbitsky \cite{Verbitsky (2007)} has shown that for every $SL(n,\H)$-manifold $(M,I,J,K)$, the associated complex manifold $(M,I)$ has a holomorphically trivial canonical bundle. In the reference \cite{Ivanov}, a comprehensive and indispensable condition is also presented for a manifold to qualify as an $SL(n,\mathbb{H})$-manifold.

As we have elaborated earlier, it is an established fact that every compact Hermitian manifold admits the existence of at least one Gauduchon metric. Within the realm of hyperhermitian geometry, the analogous result retains its validity and applicability to any compact hyperhermitian manifold. More specifically,
\begin{theorem}\cite{Alesker-Shelukhin (2013)}
Let $(M,I,J,K)$ be a compact hyperhermitian manifold endowed with a quaternionic Hermitian metric $g$. Then, there exists a unique, up to a positive multiplicative constant, nowhere vanishing form $\Theta\in \Lambda^{2n,0}$ on $(M,I)$ which is $q$-real positive such that
\begin{equation*}
\partial\partial_{J}(\Omega^{n-1}\wedge \Bar{\Theta})=0.
\end{equation*}
\end{theorem}

Recently, a profound revelation has been made in \cite{GLV17} regarding the intricate interplay between Gauduchon metrics and their quaternionic analogues, particularly within the framework of $SL(n,\mathbb{H})$ manifolds, shedding light on the fundamental connections between these two fundamental geometric constructs.
 
\begin{theorem}\cite{GLV17}
Let $(M,I,J,K)$ be a compact  $SL(n,\mathbb{H})$ manifold endowed with a quaternionic Hermitian metric $g$. Then, $g$ is  quaternionic Gauduchon   if and only if  $g$  is conformal equivalent to  a Gauduchon metric on the complex manifold $(M,I)$.
\end{theorem}

To streamline subsequent discussions, it is imperative to introduce  the following definition of quaternionic Aeppli cohomology group, which pertains specifically to the cohomology classes of $(2p,0)$-forms on $M$.
\begin{definition}\label{aeppli}
For each integer $p $ ranging from 1 to $n$, we denote by $\Lambda_I^{2p,0}(M)$  the space comprising all smooth $(2p,0)$-forms on $M$ with respect to the complex structure $I$. The quaternionic Aeppli cohomology group, denoted as $H^{2p,0}_{A}(M)$, is defined to be the group
\begin{equation*}
H^{2p,0}_{A}(M)=\frac{\{\phi\in \Lambda_{I}^{2p,0}(M)\mid \partial\partial_J\phi=0 \}}{\partial\Lambda_{I}^{2p-1,0}(M)+\partial_J\Lambda_{I}^{2p-1,0}(M)}\,.   \end{equation*}    
\end{definition}
 
In the reference \cite{GLV17}, it has been  demonstrated that for every integer $p$, the quaternionic Aeppli cohomology group $H^{2p,0}_{A}(M)$ constitutes a finite-dimensional subgroup within  $\Lambda_{I}^{2p,0}(M)$.

It is a fundamental and intriguing question to investigate whether an adaptation of Gauduchon's celebrated volume form conjecture, tailored to the realm of quaternionic geometry, can be formulated and subsequently applied to any given hyperhermitian manifold. 

\begin{question}\label{version-1}
Suppose  $(M,I,J,K,g,\Omega_{0})$ is a connected compact $SL(n,\mathbb{H})$-manifold.  Let $\Omega$ be a quaternionic Gauduchon metric on $M$. Then for every smooth real-valued $F$, there exist a unique Gauduchon metric $\tilde{\Omega}$ and a unique constant $b$ such that
\begin{equation}\label{same aeppli}
\begin{cases}
& [\tilde{\Omega}^{n-1}]=[\Omega^{n-1}] \qquad  \mathrm{in }\, H^{2n-2,0}_{A}(M) \\[2mm]
&\tilde{\Omega}^{n}=e^{F+b}\Omega_{0}^{n}.
\end{cases}    
\end{equation}
\end{question}

To rephrase the aforementioned equality \eqref{same aeppli} in accordance with Definition \ref{aeppli}, we will endeavor to find a new HKT form, denoted as  $\tilde{\Omega}$, that satisfies the following
\begin{equation}\label{Ae coho-1}
\tilde{\Omega}^{n-1}=\Omega^{n-1}+\frac{1}{2}\left(\partial(\partial_J\varphi\wedge \Omega_{0}^{n-2})-\partial_J(\partial\varphi\wedge \Omega_{0}^{n-2})\right) 
\end{equation}
for some smooth potential function $\varphi$. 
In the work by Verbitsky \cite{Verbitsky (2002)}, it has been established that the operators $\partial$ and $\partial_J$ 
exhibit anticommutativity, i.e.
\[\partial\partial_J=-\partial_J\partial.\]
Therefore, the equality \eqref{Ae coho-1} can be  reformulated as:
\begin{equation*}\label{Ae coho-2}
\tilde{\Omega}^{n-1}=\Omega^{n-1}+\partial\partial_J\varphi\wedge \Omega_{0}^{n-2}+\frac{1}{2}(\partial\varphi\wedge \partial_J\Omega_{0}^{n-2}-\partial_J\varphi\wedge \partial\Omega_{0}^{n-2})\,.
\end{equation*}

We proceed by introducing the definition of the Hodge star operator, which is defined with respect to  $\Omega_{0}$. To that end, it is sufficient to exploit its actions specifically  on the $q$-real $(2n-2)$-forms residing on the manifold $M$. At a given point, we consider the inner product of two forms belonging to $\Lambda^{p,0}_I(M)$, denoted as $\langle \cdot,\cdot \rangle_{g}$, and its definition is given by
\[
\langle \alpha,\beta\rangle_g=\frac{1}{p!} g^{r_1 \bar s_1}\cdots g^{r_p \bar s_p} \alpha_{r_1 \cdots r_p} \overline{\beta_{s_1\cdots s_p}}\,, \qquad \text{for every } \alpha,\beta \in \Lambda^{p,0}_I(M)\,,
\]
where any $(p,0)$-form $\alpha $ is locally written as
\[
\alpha=\frac{1}{p!}\alpha_{r_1 \cdots r_p}dz^{r_1}\wedge \dots \wedge dz^{r_p}
\]
and $(g^{r\bar s})$ is the inverse of the Hermitian matrix $(g_{r\bar s})$ induced by the $I$-Hermitian metric $g$. Then, the Hodge star operator $*\colon \Lambda_I^{2n-2,0}(M) \to \Lambda_I^{2,0}(M)$ is explicitly defined as
\[
\alpha \wedge *\beta= \langle \alpha,\beta\rangle_{g}\frac{\Omega_{0}^n}{n!}\,,\qquad \text{for } \alpha,\beta \in \Lambda_I^{2n-2,0}(M)\,.
\]
As demonstrated by Gentili and the author in \cite{GZ22}  that
\begin{equation*}
\Big(\frac{\tilde{\Omega}^{n}}{\Omega_{0}^{n}}\Big)^{n-1}=  \frac{\Big(*(\tilde{\Omega}^{n-1})\Big)^{n}}{\Big(*(\Omega_{0}^{n-1})\Big)^{n}}=\frac{\Big(*\Big(\frac{\tilde{\Omega}^{n-1}}{(n-1)!}\Big)\Big)^{n}}{\Omega_{0}^{n}}\,.
\end{equation*}
Consequently, Equation \eqref{same aeppli} is transformed into
\begin{equation}\label{Monge-Ampere-2}
\Big(*\Big(\frac{\tilde{\Omega}^{n-1}}{(n-1)!}\Big)\Big)^{n}=e^{(n-1)(F+b)}\Omega_{0}^{n}\,.  
\end{equation}
It has been established in \cite{GZ22} that for every $\varphi\in C^{\infty}(M,\R)$
\begin{equation}\label{Hodge3}
\frac{1}{(n-1)!} * \left( \partial \partial_J \varphi \wedge \Omega_{0}^{n-2} \right)=\frac{1}{n-1} \left[ (\Delta_{g} \varphi) \Omega_{0}-\partial \partial_J \varphi \right]\,,
\end{equation}	
where the canonical Laplacian $\Delta_{g}$ is defined by 
\[
\Delta_{g}\varphi:=\frac{n\partial\partial_J \varphi\wedge\Omega_{0}^{n-1}}{\Omega_{0}^{n}}\,.
\]
Given that the HKT metric $\Omega$ is positive on $M$, there correspondingly exists a positive definite form 
\begin{equation}\label{hat omega 0}
H:=\frac{*(\Omega^{n-1})}{(n-1)!}\in \Lambda_{I}^{2,0}(M)
\end{equation}
that is intimately associated with it. Let us introduce a $q$-real $(1,1)$-form
\begin{equation}\label{Z varphi}
Z[\varphi]:=\frac{1}{2(n-1)!}*\{\partial\varphi\wedge \partial_J\Omega_{0}^{n-2}-\partial_J\varphi\wedge \partial\Omega_{0}^{n-2}\}
\end{equation}
which is explicitly defined in terms of linear dependencies on the first-order derivatives of $\varphi$. From the relations stated in Equations \eqref{Hodge3}, \eqref{hat omega 0} and \eqref{Z varphi}, it is immediately evident that the form
\begin{equation}\label{hodge star}
*\Big(\frac{\tilde{\Omega}^{n-1}}{(n-1)!}\Big)=  H+\frac{1}{n-1} \left[ (\Delta_{g} \varphi) \Omega_{0}-\partial \partial_J \varphi \right]+Z[\varphi]
\end{equation}
determines a positive definite element in $\Lambda_{I}^{2,0}(M)$, owing to the positivity of $\tilde{\Omega}^{n-1}$. 
\begin{definition}
A function $\varphi\in C^2(M,\mathbb{R})$ is quaternionic $(n-1)$-PSH with respect to $\hat{\Omega}$ and $\Omega_{0}$ if  $$ H+\frac{1}{n-1}[(\Delta_g \varphi) \Omega_{0}-\partial \partial_J \varphi]+Z[\varphi]\in \Lambda_{I}^{2,0}(M) $$ is positive definite.   
\end{definition}

Having the equality given in Equation \eqref{hodge star} at our disposal, we can restate \eqref{Monge-Ampere-2}  as the following quaternionic $(n-1)$-form Monge-Amp\`ere equation
\begin{equation}\label{n-1 MA}
\begin{cases}
\ \Big(H+\frac{1}{n-1}[(\Delta_g \varphi) \Omega_{0}-\partial \partial_J \varphi]+Z[\varphi] \Big)^{n} = {\rm e}^{(n-1)(F+b)}\Omega_{0}^{n}\,,  \\[2mm]
\ H+\frac{1}{n-1}[(\Delta_g \varphi) \Omega_{0}-\partial \partial_J \varphi]+Z[\varphi] > 0\,.
\end{cases}
\end{equation}
We remark that Equation \eqref{n-1 MA} serves as the quaternionic counterpart of the complex Monge-Amp\`ere equation for $(n-1)$-PSH functions, originally introduced and extensively analyzed by Fu, Wang and Wu in \cite{FWW10,FWW15}.  Further related contributions can be discovered in the literature \cite{KLW22,STW17,TW17,TW19} and the reference therein.

Consider a closed Hermitian manifold $(M,I,g)$   of complex dimension $m=2n$ (where $n\in \N$). Assume that $J$ is an another complex structure defined on $M$  that is compatible with respect to the metric $g$, satisfying the orthogonality condition  $$JI+JI=0.$$  Given the K\"ahler form $\omega$ associated with the metric $g$ and $\omega_{h}$ another specified Hermitian metric on $M$, we introduce the following potential space
\[
\begin{split}
&\PSH_J(M,\omega,\omega_{h})\\
&=\Big\{ \varphi\in C^2(M,\R) : \omega_{h}+\frac{1}{m-1}\big[ (\Delta_{I,g}\varphi)\omega-\frac{1}{2}(\ddbar \varphi-J\ddbar \varphi)\big]+W[\varphi]>0  \Big\}\,.
\end{split}
\]
Here, $\Delta_{I,g}$ represents the canonical Laplacian defined on $(M,I,g)$, which is given by $$\Delta_{I,g}\varphi=\frac{m(\ddbar \varphi-J\ddbar \varphi)\wedge \omega^{m-1}}{2\omega^{m}}\,,$$
and $W[\varphi]$ is a Hermitian tensor  that depends linearly on the first-order derivatives of $\varphi$, namely
\[
W[\varphi]:=\frac{\sqrt{-1}}{2}(\varphi_{i}\bar{\beta}_{j}+\beta_{i}\varphi_{\bar{j}})dz_{i}\wedge d\bar{z}_{j}
\]
for a smooth $(1,0)$ form $\beta$  which is locally expressed as $\beta=\beta_{j}dz_{j}$. For the sake of clarity and convenience, let us denote
\begin{equation*}\label{tilde g}
\tilde{g}_{i\bar j}=h_{i\bar j}+\frac{1}{m-1}[(\Delta_{I,g}\varphi) g_{i\bar j}-\frac{1}{2}(\varphi_{i\bar j}+J^{\bar{l}}_{i}J^{k}_{\bar j}\varphi_{k\bar l })]+\frac{1}{2}(\varphi_{i}\bar{\beta}_{j}+\beta_{i}\varphi_{\bar{j}})\,,   
\end{equation*}
where $h_{i\bar j}$ is locally defined by $\omega_{h}=\sqrt{-1}h_{i\bar j}dz_i\wedge d\bar{z}_j$.

In order to solve the Equation \eqref{n-1 MA}, we need to to obtain the a priori estimates for the following complex $(m-1)$-form Monge-Amp\`ere equation on $(M,I,g)$ 
\begin{equation}\label{new n-1 MA}
\begin{cases}
\ \det(\tilde{g}_{i\bar j})=e^{G}\det(g_{i\bar j})\,,  \\[2mm]
\ \sup_M\varphi=0, \quad  \varphi\in \PSH_{J}(M,\omega,\omega_{h})\,,
\end{cases}
\end{equation}
where $G$ is a given smooth function on $M$. Intuitively, the uniform estimate occupies a pivotal position in the investigation of fully nonlinear elliptic equations, as it serves as a fundamental prerequisite for establishing higher-order regularities.

Inspired by the groundbreaking concept introduced by Guo-Phong-Tong \cite{GPT} in their recent work on the auxiliary Monge-Amp\`ere equation, we present a uniform estimate for \eqref{new n-1 MA} that is solely dependent on the p-entropy of the prescribed volume form and the background data, offering a novel perspective on the problem. Specifically, the main objectives of the present paper are as follows: 
\begin{theorem}\label{main}
Suppose $(M,I,J,K,g,\Omega_{0})$ is a connected compact $SL(n,\mathbb{H})$-manifold.  Given any constant $p>2m$. Let $\varphi\in \PSH_J(M,\omega,\omega_{h})$ be a smooth solution for the $(m-1)$-form Monge-Ampere equation \eqref{new n-1 MA}. Then there exists a constant $C>0$   depends only on $n,\,\omega,\, \omega_{h},\, p$ and $\|e^{2G}\|_{L^1(\log L)^p}$ such that 
\begin{equation*}
\|\varphi\|_{L^{\infty}} \leq C\,.
\end{equation*}
\end{theorem}	
	
It is worth emphasizing that the attainment of higher-order regularity for the Equation \eqref{new n-1 MA} poses a significantly greater challenge in its general form, partly attributed to the intricacies arising from the derivatives of $J$. In the  specific instance where the term $Z$ in \eqref{n-1 MA} vanishes, the question of solvability has been resolved by G. Gentili and the author \cite{GZ22} for flat hyperk\"ahler manifolds,  and additionally  by Fu-Xu-Zhang  \cite{FXZ22}  for the more general class of hyperk\"ahler manifolds.

The rest of paper is organized as follows. In Section \ref{Auxiliary and exp}, we commence by reducing the proof of Theorem \ref{main} to the proof of Proposition \ref{lower bound} outlined below. Subsequently, we delve into auxiliary real Monge-Amp\`ere equations and the exponential integrability of their solutions. In Section \ref{Comparison lemma and Trudinger-like inequality}, we will derive a pivotal comparison lemma alongside a Trudinger-type inequality, both of which will play a fundamental role in our subsequent analysis. In Section \ref{proof of main proposition}, we establish an $L^{q}$ (for $q>0$) estimate by utilizing the standard weak Harnack inequality. This estimate is then leveraged to prove Proposition \ref{lower bound}. \bigskip

\noindent 
{\bf Acknowledgements.}  The  author is  grateful to his thesis advisor  Professor Xi Zhang for his inspirational discussions and helpful suggestions. This paper is supported by Initial Scientific Research Fund of Young Teachers in Hefei university of Technology of No.JZ2024HGQA0119.

\section{Auxiliary real Monge-Amp\`ere equations and  exponential integrability}\label{Auxiliary and exp}

Assume that $\varphi\in C^2(M,\mathbb{R})$ serves as a solution to the Equation \eqref{new n-1 MA}. To establish Theorem \ref{main}, it suffices to derive a lower bound for  $\varphi$. By compactness, we can assume that the function $\varphi$ attains its minimum value at a certain point  on $M$, denote as $z_{\mathrm{min}}$, i.e. $$\varphi(z_{\mathrm{min}})=\inf_{z\in M}\varphi(z).$$ 
Hence, in the subsequent sections, we are expected to exploit the lower bound for $\varphi(z_{\mathrm{min}})$  in the manner as follows
\begin{proposition}\label{lower bound}
Given $p>2m$, there exists a positive constant $C$ which depends only on $m$, $\omega$, $\omega_{h}$, $p$ and $\|e^{2G}\|_{L^1(\log L)^p}$  such that
\[
\qquad \varphi(z_{\mathrm{min}})\geq -C.
\]
\end{proposition}

The proof of Proposition \ref{lower bound} is somewhat extensive and necessitates additional efforts, we will carry out it at the end of present paper.

\subsection{Auxiliary real Monge-Amp\`ere equations}	

Let $V$ be an open subset that is isomorphic to an open subset within $\mathbb{C}^{2n}$, with the property that $z_{\mathrm{min}} $ lies within $V$. Then, there exists a uniform positive constant $r_0$ for which the ball $B=B(z_{\mathrm{min}},2r_0)$ is contained entirely within $ V$. If necessary, Shrink the value of $r_0$ to ensure that the following inequality holds on $\overline{B}$:
\begin{equation}\label{coordinates}
\frac{1}{C_{1}}\,\omega_{_{\C^m}}\leq \omega_{h}\leq C_{1}\, \omega_{_{\C^m}}\,, \qquad \omega_{\C^m}=\ddbar |z|^2\,.
\end{equation}
Here, $C_{1}$ is a positive constant that depends  solely on $\omega_{h}$.

We introduce a small, positive constant $c_0$, whose specific value will be determined later.	For any $s $ within the interval  $ (0,c_0r_0^{2})$, we consider the following auxiliary function
\begin{equation*}
\varphi_s(z)=\varphi(z)-\varphi(z_{\mathrm{min}})+c_0|z-z_{\mathrm{min}}|^2-s, \qquad z\in \overline{B}.
\end{equation*}
Firstly, observe that for any $z\in B\setminus B(z_{\mathrm{min}},r_{0})$, we have the inequality $\varphi_s(z)\geq c_0r_0^{2}-s>0$. Therefore, the sub-level set $B_s$, which is defined by
\begin{equation*} 
B_s=\{z\in B: \varphi_s(z)<0\},
\end{equation*}
is an open subset of $B$. Furthermore, $B_s$ is non-empty and has a compact closure within $B$.

For each index $i=1, \cdots, m$, we express the complex number $z_{i} $ as $ x_{2i-1}+\sqrt{-1}x_{2i}$, ensuring that the set $\{x_{a}\}_{a=1}^{2m}$ constitutes a local real coordinate chart. In this local real coordinate chart, we denote the standard volume form in $\mathbb{R}^{2n}$ by ${\rm d x}$. To streamline upcoming discussions and for the sake of notational conciseness,   we abbreviate the expression
\[
{\rm d}\sigma= (\det g_{i\bar{j}})^{2}{\rm dx}
\]
and introduce two pivotal quantities that will be recurrently used
\begin{equation}\label{two quantity}
\Phi_{s}=\int_{B_s}(-\varphi_s)e^{2G} {\rm d}\sigma,  \qquad   \Psi_{s}=\int_{B_s}e^{2G}{\rm d}\sigma\,,
\end{equation}
for any $s\in (0,c_0r_0^{2})$.

\begin{lemma}\label{Phi A}
For every $0<t<s<c_0r_0^{2}$, we have 
\begin{equation*}
t\Psi_{s-t}\leq \Phi_{s}\,.
\end{equation*}
\end{lemma}
\begin{proof}
The proof is somewhat straightforward, we are about to provided a sketch here for the sake of completeness. According to the definition of $B_s$,  it follows that $\varphi_{s-t} $ is strictly negative (i.e. $t<-\varphi_s$) on $B_{s-t}$. Therefore,
\[
t\int_{B_{s-t}}e^{2G}{\rm d}\sigma\leq \int_{B_{s-t}} (-\varphi_s)e^{2G}{\rm d}\sigma \leq \int_{B_s}(-\varphi_s)e^{2G}{\rm d}\sigma\,.
\]
This shows our desired inequality.
\end{proof}

Select a sequence of smooth and strictly positive functions  $\{\tau_k(x)\}_{k=1}^{\infty}$ defined on $\mathbb{R}$ that converges pointwise to the function $\tau_{\infty}(x):=x\cdot \chi_{R_{+}}(x)$ as $k\rightarrow \infty$, where $\chi_{R_{+}}(x)$ denotes the indicator function for the non-negative real numbers. It is convenient to select a specific candidate for $\tau_k(x)$ that possesses the following properties:
\[\tau_{k}(x)=
\left\{\begin{array}{ll}
x+\frac{1}{k}\qquad &\mathrm{for}~x\in [0,+\infty)\,, \\[1mm]
\frac{1}{2k}\qquad &\mathrm{for}~x\in (-\infty,-\frac{1}{k}]\,, \\[1mm]
\end{array}\right.
\]
as well as
\[
\frac{1}{2k} \leq \tau_{k}(x)\leq \frac{1}{k} \qquad  \mathrm{for}~ x\in [-\frac{1}{k},0].
\]

As outlined in \cite{GPT,GP22}, we proceed by letting $\varphi_{s,k}$ represent a family of solutions belonging to the class $C^{\infty}(\overline{B})$ for the following nondegenerate Dirichlet problems  associated with the auxiliary real Monge-Amp\`ere equations 
\begin{equation}\label{definition of varphi sk}
\left\{ \begin{array}{ll}
\varphi_{s,k}\in C^{\infty}(\overline{B}) \,\, \textrm{strictly convex}, \\[2mm]
\det\left( \frac{\partial^2\varphi_{s,k}}{\partial x_{a}\partial x_{b}}\right)=\frac{1}{\Phi_{s,k}}\tau_k(-\varphi_s)e^{2G}(\det g_{i\bar{j}})^{2} \quad \textrm{in} \, B, \\[2mm]
\varphi_{s,k}=0 \quad \textrm{on}  \quad \partial B,\\[2mm]
\end{array}\right.
\end{equation}
where $\Phi_{s,k}$ denotes a normalization constant, which is specifically defined as follows:
\[\Phi_{s,k}:=\int_{B}\tau_k(-\varphi_s)e^{2G}{\rm d}\sigma\,.\]
It is well-established that the uniqueness and existence of solutions $\varphi_{s,k}$ for \eqref{definition of varphi sk} are guaranteed by the renowned theorem of Caffarelli-Nirenberg-Spruck \cite{CKS84}. According to the definition of $\Phi_{s,k}$,  it becomes evident that 
\begin{equation}\label{unit volume}
\int_{B}\det\Big( \frac{\partial^2\varphi_{s,k}}{\partial x_{a}\partial x_{b}}\Big){\rm dx}=1\,,    
\end{equation}
as well as
\begin{equation}\label{phi sk and phi s}
\Phi_{s,k} \longrightarrow \Phi_{s} \qquad \mathrm{as} \,\, k\rightarrow \infty. 
\end{equation} 

Observe that
\[
-\varphi_{s}(z)=\varphi(z_{\mathrm{min}})-\varphi(z)-c_{0}|z-z_{\mathrm{min}}|^2+s\leq s<c_0r_0^{2}\]
in $B_s$ for every $s<c_0r_0^{2}$. Therefore, the integrands in \eqref{two quantity} are both bounded, thereby enabling the estimation of  $\Phi_{s}$ and $\Psi_{s}$ in terms of $\|e^{2G}\|_{L^{1}}$. This implies that, given a sufficiently large  $k$,  the boundedness of $\Phi_{s,k}$ can be inferred from \eqref{phi sk and phi s}, as well as the previous considerations. 

To derive the comparison Lemma \ref{comparison lemma} below, the subsequent local estimates pertaining to $\varphi_{s,k}$ are of paramount importance.
\begin{lemma}\label{c0 and c1}\cite[Lemma 11]{GP22} There exists a dimensional constant $C(m)$ such that
\begin{equation*}
\inf_{B(z_{\min},2r_{0})}\varphi_{s,k}\geq -C(m)r_{0}\,, \quad \sup_{B(z_{\min},r_{0})}|\nabla^{Ob} \varphi_{s,k}|\leq C(m)\,.
\end{equation*}
    
\end{lemma}
\subsection{Exponential integrability}
Exponential integrability has been demonstrated to be a very powerful tool in complex geometry, particularly the systematic investigations for the complex Monge-Amp\`ere equations by Guo-Phong-Tong \cite{GPT} very recently.

Assume that $\Omega$ is a bounded, smooth, and pseudoconvex domain within the complex space $\mathbb{C}^{m}$. Let  $\varphi\in \PSH(\Omega)\cap C^{2}(\Omega)\cap C^{0}(\bar{\Omega})$ be a function that possesses the properties of vanishing on the boundary, i.e. $\varphi|_{\partial\Omega}=0$ and 
\begin{equation*}
    \mathcal{M}(\varphi):=\int_{\Omega}(\ddbar \varphi)^{m}=1\,.
\end{equation*}
A celebrated theorem by Ko{\l}odziej, as presented in \cite{Ko98}, shows that there exists positive constants $C'$ and $\alpha'$, which are solely dependent on the geometric characteristics of the domain $\Omega$, such that
\begin{equation}\label{alpha invariant 1}
\int_{\Omega}\exp\{-\alpha'\varphi\}\omega^{m}\leq C'\,. 
\end{equation}
It is noteworthy that Ko{\l}odziej's proof of \eqref{alpha invariant 1} heavily relies on the the framework of pluripotential theory, whereas an elegant and purely classical PDE-based proof has been provided by Wang-Wang-Zhou in \cite{WWZ21}.

In contrast to the complex setting, the results pertaining to exponential integrability are significantly more robust in the real case. For a comprehensive account of these findings, we recommend referring to the literature \cite{Wang}. Consequently, we can formulate the following:
\begin{lemma}\cite{Wang}
Let $\varphi_{s,k}$ be a solution of Equation \eqref{definition of varphi sk}. Then there exist positive constants $C_{2}$ and $\alpha$ that depending only on the geometric data such that
\begin{equation}\label{alpha invariant}
 \int_{B}\exp\{-\alpha\varphi_{s,k} \}{\rm d}\sigma\leq C_{2}\,.   
\end{equation}
\end{lemma}

\section{Comparison lemma and Trudinger's  inequality}\label{Comparison lemma and Trudinger-like inequality}
\subsection{A comparison lemma}
Following the approach adopted in \cite{STW17} for demonstrating the Gauduchon conjecture, we aim to introduce a highly instrumental tensor $\hat{\Theta}$ with
\[
\hat{\Theta}^{i\bar j}=\frac{1}{m-1}\Big\{\mathrm{tr}_{\tilde{\omega}}\omega\cdot g^{i\bar j}-
\frac{1}{2}\Big(\tilde{g}^{i\bar j}+\tilde{g}^{k\bar l}J_{k}^{\bar j}J_{\bar l}^{i} \Big) 
\Big\}\,.
\]
Regarding the tensor $\hat{\Theta} $, we can establish the following crucial inequality.
\begin{lemma}\label{positive det2}
At every point, the tensor  $\hat{\Theta}$ is positive definite, and it satisfies 
\begin{equation*}\label{det hat Theta ij}
\det(\hat{\Theta}^{i\bar j})\geq e^{-G}\det(g^{i\bar j})\,.
\end{equation*}
\end{lemma}
\begin{proof}
Let $\lambda_1, \cdots,\lambda_{m}$ denote the eigenvalues of the tensor $\tilde\omega $ with respect to the background metric $\omega$. In the vicinity of a specified point $z_{\mathrm{min}}$, there exists a local $I$-holomorphic coordinate system $(z_{1}, \cdots,z_{m})$ such that at $z_{\mathrm{min}}$,  the components of the metrics of $g$ and $\tilde{g}$ satisfy $g_{i\bar j}=\delta_{ij}$ and $\tilde g_{i\bar j} =\lambda_i\delta_{ij}$. As a consequence,
\begin{equation}\label{hat Theta ij}
\begin{split}
\hat{\Theta}^{i\bar j}|_{z_{\mathrm{min}}}=&\frac{1}{m-1}\Big\{ \sum_{i}\frac{1}{\lambda_i}\delta_{ij}-\frac{1}{2}\Big(\frac{1}{\lambda_i}\delta_{ij}+\frac{1}{\lambda_{k}}\delta_{kl}J_{k}^{\bar j}J_{\bar  l}^{i}\Big)\Big\}\\ 
=&\frac{1}{m-1}\Big(\sum_{k\neq i} \frac{1}{\lambda_{k}}\Big)\delta_{ij}\,.
\end{split}
\end{equation} 
Applying the arithmetic and geometric mean inequality, we can observe that
\[
\det(\hat{\Theta}^{i\bar j})=\frac{1}{(m-1)^m}\prod_{i=1}^{m}\Big(\sum_{k\neq i} \frac{1}{\lambda_{k}}\Big)\geq \prod_{i=1}^{m} \frac{1}{\lambda_i}=e^{-G}\det(g^{i\bar j})\,.
\]
This ends our proof of the lemma.
\end{proof}
	
\begin{lemma}\label{L varphi}
Consider the linearized operator $\mathcal{L}$ defined as $\mathcal{L}v=\hat{\Theta}^{i\bar j}v_{i\bar j},$
we have the following
\begin{equation}\label{l varphi}
\mathcal L\varphi=m-\tr_{\tilde    \omega}\omega_{h}-\tr_{\tilde\omega}W[\varphi]\,.
\end{equation}
\end{lemma}
\begin{proof}
A direct calculation yields that 
\[
\begin{split}
\tr_{\tilde \omega}\Big(\tilde\omega-\omega_{h}-W[\varphi]\Big) = & \frac{1}{2} \tr_{\tilde \omega}\Big(\ddbar \varphi-J\ddbar \varphi  \Big) 
=  \frac{1}{2}\Big(\tilde{g}^{i\bar j}\varphi_{i\bar j} +\tilde{g}^{i\bar j}J^{\bar{l}}_{i}J^{k}_{\bar j}\varphi_{k\bar l }) \Big)\\
=& \frac{1}{2}\Big(\tilde{g}^{i\bar j}+\tilde{g}^{k\bar l}J_{k}^{\bar j}J_{\bar l}^{i}\Big) \varphi_{i\bar j} = \hat{\Theta}^{i\bar j}\varphi_{i\bar j}\,.
\end{split}
\]
Then the proof is completely.
\end{proof}

In the subsequent discussion, our focus will be directed towards a comparative analysis of the functions $\varphi_{s,k}$ and $\varphi_s\,$. This will be pretty crucial in our proof of Proposition \ref{lower bound} at the end of this paper.
\begin{lemma}\label{comparison lemma}
Let $\epsilon=\Phi_{s,k}^{\frac{1}{2m+1}}(\frac{m}{2m+1})^{\frac{2m}{2m+1}}$. The inequality
\begin{equation}\label{comparison}
-\varphi_s\leq \epsilon (-\varphi_{s,k}+\varepsilon)^{\frac{2m}{2m+1}} \quad \mathrm{in} \, B
\end{equation}
holds true for a specific constant $\varepsilon>0$ depending only on $m,$ $\Phi_{s,k}$ and $\beta$.
\end{lemma}
\begin{proof}
Firstly, we direct our attention towards the following auxiliary function defined as
\[\eta=-\epsilon \Gamma(\varphi_{s,k})-\varphi_s,\]
where  $\Gamma:(-\infty,0]\rightarrow \mathbb{R}_{+}$ is  a smooth, non-increasing, concave function, whose precise form will be specified subsequently. Notably, this function $\eta$ is designed such that its maximum value on $\overline{B}$ satisfies $\max_{\overline{B}}\eta\leq 0$. 

Upon unraveling the definition of $\varphi_s$, it becomes immediately apparent that $\eta\leq 0$ holds true within the set $\overline{B} \setminus B_s$. Hence, we can  reasonably assume the maximum of $\eta$ is achieved at an interior point within  $B_s\,$, which we denote as $z_{\max}$. Utilizing the maximum principle in conjunction with the fact that $\hat{\Theta}^{i\bar j}$ is positive definite, we can deduce that the value of $L\eta$ at the point $z_{\max}$ satisfies $\mathcal L\eta|_{z_{\max}}\leq 0$.

Adopting the same local coordinate system in the vicinity of $z_{\max}$ as was utilized in the proof of Lemma \ref{det hat Theta ij},  we can subsequently derive that
\begin{equation*}
\begin{split}
\mathcal L|z-z_{\mathrm{min}}|^2|_{z_{\max}}&=\sum_{i}\tilde{\Theta}^{i\bar i}=\frac{1}{m-1}\sum_{i}\Big\{\mathrm{tr}_{\tilde{\omega}}\omega\cdot g^{i\bar i}-\frac{1}{2}\Big(\tilde{g}^{i\bar i}+\tilde{g}^{k\bar l}J_{k}^{\bar i}J_{\bar l}^{i} \Big)\Big\}\\
&= \frac{1}{m-1} \sum_{i}\Big\{\mathrm{tr}_{\tilde{\omega}}\omega\cdot g^{i\bar i}-
\tilde{g}^{i\bar i}\Big\}\\
&=\frac{1}{m-1}\{\tr_{\tilde\omega}\omega\cdot \tr_{\omega}\omega_{_{\C^m}}-\tr_{\tilde\omega}\omega_{_{\C^m}}\}|_{z_{\max}}\,.   
\end{split}
\end{equation*}
Taking into account the coordinates as specified in \eqref{coordinates}, we can further deduce that
\begin{equation}\label{L Z squre}
\mathcal L|z-z_{\mathrm{min}}|^2|_{z_{\max}}\leq  \frac{C_{1}}{m-1}\tr_{\tilde \omega}\omega|_{z_{\max}}\cdot \tr_{\omega}\omega_{h}|_{z_{\max}}.
\end{equation}
In order to facilitate the  subsequent discussions, we introduce the following notations
\begin{equation*}
  [f]_{i}=\frac{\partial f}{\partial z_{i}}\,, \qquad  [f]_{\bar j}=\frac{\partial f}{\partial \bar{z}_{j}}\,, \qquad [f]_{i\bar j}=\frac{\partial^2 f}{\partial z_{i}\partial \bar{z}_{j}}
\end{equation*}
for any function $f$ defined on $M$ that is of class $C^2$, and for any indices $i,j=1, \cdots, m$. Consequently,  at the point $z_{\max}$, we obtain
\begin{equation}\label{first der}
    [\varphi]_{i}=\epsilon \Gamma'[\varphi_{s,k}]_{i}-\epsilon c_0 (z_{i}-z_{\min,i})\,,
\end{equation}
as well as
\begin{equation*}
\begin{split}
0\geq & \mathcal L\eta|_{z_{\max}}
= -\epsilon \mathcal{L}\Big(\Gamma(\varphi_{s,k})\Big) -\mathcal L\varphi-c_0\mathcal L|z-z_{\mathrm{min}}|^{2} \\
=  &-\epsilon \Gamma'\hat{\Theta}^{i\bar j}[\varphi_{s,k}]_{i\bar j}-\epsilon \Gamma''\hat{\Theta}^{i\bar j}[\varphi_{s,k}]_{i}[\varphi_{s,k}]_{\bar j} -m+\mathrm{tr}_{\tilde\omega}\omega_{h}+\tr_{\tilde\omega}W[\varphi]-c_0\mathcal L|z-z_{\mathrm{min}}|^{2} \\
\geq &-\epsilon m \Gamma' (\det(\hat{\Theta}^{i\bar j}))^{\frac{1}{m}} (\det[\varphi_{s,k}]_{i\bar j})^{\frac{1}{m}}-m+\tr_{\tilde{\omega}} \{\omega_{h}-\frac{c_0C_{1}}{m-1}(\tr_{\omega}\omega_{h})\omega+W[\varphi]\}\\
\geq &-\epsilon m \Gamma' \frac{(-\varphi_s)^{\frac{1}{2m}}}{2\Phi_{s,k}^{\frac{1}{2m}}}-m+\tr_{\tilde{\omega}} \{\omega_{h}-\frac{c_0C_{1}}{m-1}(\tr_{\omega}\omega_{h})\omega+W[\varphi]\} 
\end{split}
\end{equation*}
where the second inequality arises from the facts that $\Gamma$ is a smooth, non-increasing, and concave function, the arithmetic and geometric mean inequality, and the inequality \eqref{L Z squre}, the last inequality we have used Lemma \ref{positive det2} and the following well-known inequality
\[
2^{2m} \det\Big(\frac{\partial^2\varphi_{s,k}}{\partial z_{i}\partial z_{j}}\Big)^{2}\geq \det\Big(\frac{\partial^2\varphi_{s,k}}{\partial x_{a}\partial x_{b}}\Big)\,.
\]

Now, we proceed to  estimate $\tr_{\tilde{\omega}}W[\varphi]$ in terms of $\tr_{\tilde{\omega}}\omega_{h}$. First of all, the equality given in \eqref{first der} entails that
\begin{equation}\label{trace 1}
\begin{split}
\tr_{\tilde{\omega}}W[\varphi]&=\mathrm{Re}\Big(\tilde{g}^{i\bar{j}}[\varphi]_{i}\bar{\beta}_{j}\Big)\\
&=\epsilon\Gamma'\mathrm{Re}\Big(\tilde{g}^{i\bar{j}}\bar{\beta}_{j}[\varphi_{s,k}]_{i}\Big)- \epsilon c_0 \mathrm{Re}\Big(\tilde{g}^{i\bar{j}}\bar{\beta}_{j}(z_{i}-z_{\min,i})\Big).   
\end{split}
\end{equation}
On the one hand, suppose there exists a positive constant $C_{\beta}$ that depends  solely on $\beta$ and $\omega_{h}$ such that
\begin{equation*}
\sqrt{-1}\beta\wedge \bar{\beta}\leq C^{2}_{\beta}\omega_{h}\,.
\end{equation*}
Furthermore, the regularity of $\varphi_{s,k}$ as stated in Lemma \ref{c0 and c1}, combined with the Cauchy-Schwarz inequality, ensures that
\begin{equation}\label{first term}
\begin{split}
\epsilon\Gamma'\mathrm{Re}\Big(\tilde{g}^{i\bar{j}}\bar{\beta}_{j}[\varphi_{s,k}]_{i}\Big)&\geq   \epsilon\Gamma'\Big|\mathrm{Re}\Big(\tilde{g}^{i\bar{j}}\bar{\beta}_{j}[\varphi_{s,k}]_{i}\Big)\Big|\\
  &\geq  \epsilon\Gamma'(0)\Big(\tr_{\tilde{\omega}}\sqrt{-1}\beta\wedge\bar{\beta}\Big)^{\frac{1}{2}}\Big(\tr_{\tilde{\omega}}\sqrt{-1}\partial \varphi_{s,k}\wedge\bar{\partial}\varphi_{s,k}\Big)^{\frac{1}{2}}  \\
  &\geq \epsilon\Gamma'(0)C(m)C_{\beta}\tr_{\tilde{\omega}}\omega_{h}\,.
\end{split}
\end{equation}
On the other hand, since $z_{\max} $ lies within the ball  $ B(z_{\mathrm{min}},2r_0)$, it follows that
\begin{equation}\label{second term}
    - \epsilon c_0 r_0 \mathrm{Re}\Big(\tilde{g}^{i\bar{j}}\bar{\beta}_{j}(z_{i}-z_{\min,i})\Big)\geq -2\epsilon c_0C_{\beta}\tr_{\tilde{\omega}}\omega_{h}\,.
\end{equation}
Substituting the inequalities \eqref{first term} and \eqref{second term} into \eqref{trace 1}, we derive that
\begin{equation*}
   \tr_{\tilde{\omega}}W[\varphi]\geq  -\Big(-\epsilon\Gamma'(0)C(m)C_{\beta}+2\epsilon c_{0} r_{0}C_{\beta}\Big)\tr_{\tilde{\omega}}\omega_{h}\,.
\end{equation*}

Now, we choose a sufficiently small positive constant $c_0$ such that
\[
\frac{c_0C_{1}}{m-1}(\tr_{\omega}\omega_{h}) \omega\leq \frac{1}{3}\omega_{h}, \qquad 2\epsilon c_0C_{\beta}\leq \frac{1}{3}\,.
\]
Taking into consideration all these factors, it becomes immediately apparent that
\begin{equation*}
0\geq -\epsilon m \Gamma' \frac{(-\varphi_s)^{\frac{1}{2m}}}{2\Phi_{s,k}^{\frac{1}{2m}}}-m+\frac{1}{3}\Big(1+3\epsilon\Gamma'(0)C(m)C_{\beta}\Big)\tr_{\tilde{\omega}}\omega_{h}\,.
\end{equation*}
Therefore, to guarantee that $\eta|_{z_{\max}}=-\epsilon\Gamma-\varphi_{s}\leq 0$, i.e. $-\varphi_{s}\leq \epsilon\Gamma$ at the point $z_{\max}$,  it suffices to impose the following conditions on $\Gamma$:
\begin{equation}\label{condition}
-\epsilon \Gamma'(\epsilon\Gamma)^{\frac{1}{2m}}=2\Phi_{s,k}^{\frac{1}{2m}}, \qquad \Gamma'(0)=-\frac{1}{3\epsilon C(m)C_{\beta}}.    
\end{equation}
\begin{itemize}
\item For the first equality of \eqref{condition}: the hypothesis on $\epsilon$ implies that $\Gamma$ satisfies the following ODE:
\[
\Big(\Gamma^{\frac{2m+1}{2m}} \Big)'=-\frac{2m+1}{m}\epsilon^{-\frac{2m+1}{2m}}\Phi_{s,k}^{\frac{1}{2m}}=-1.
\]
Hence, $\Gamma(\varphi_{s,k})$ can be precisely defined as
\begin{equation}\label{gamma}
\Gamma(\varphi_{s,k})=(-\varphi_{s,k}+\varepsilon)^{\frac{2m}{2m+1}}, \quad \Gamma'(0)=-\frac{2m}{2m+1}\varepsilon^{-\frac{1}{2m+1}}   \end{equation}
for some constant $\varepsilon>0$ to be specified below.  
\item For the second equality of \eqref{condition}: it is sufficient to define
\begin{equation*}
\begin{split}
\varepsilon:&=\Big(\frac{6m}{2m+1}C(m)C_{\beta}\epsilon\Big)^{2m+1}\\
&=\Big(\frac{6m}{2m+1}C(m)C_{\beta}\Big)^{2m+1}\Big(\frac{m}{2m+1}\Big)^{2m}\Phi_{s,k}.  
\end{split}
\end{equation*}
In other word,
\[
\frac{1}{3\epsilon C(m)C_{\beta}}=\frac{2m}{2m+1}\varepsilon^{-\frac{1}{2m+1}}\,.
\]
Therefore, $\varepsilon$ can be derived directly.
\end{itemize}
Consequently, our proof of the proposition is completed. 
\end{proof}

\subsection{Trudinger's inequality}
The primary objective of this subsection is dedicated to establishing the validity of the following Trudinger's inequality.
\begin{proposition}
For every $s\in (0,c_0r_0^{2})$, there exist two strictly positive constants $C_{4}$ and $\beta$, which depend only on $m$ and $r_0$, such that
\begin{equation}\label{trudinger}
\int_{B_s} \exp\Big\{\beta \frac{(-\varphi_s)^{\frac{2m+1}{2m}}}{\Phi_{s}^{\frac{1}{2m}}}\Big\}{\rm d}\sigma\leq C_{4}\,.
\end{equation}
\end{proposition}
\begin{proof}
Firstly, the inequality given in \eqref{comparison} asserts that
\[\frac{(-\varphi_s)^{\frac{2m+1}{2m}}}{\Phi_{s,k}^{\frac{1}{m}}}\leq c(m) (-\varphi_{s,k})\]
for a positive dimensional constant $c(m)$. Select a sufficient small positive constant $\beta$ such that it satisfies the condition $\beta c(m)\leq \alpha$, where $\alpha$ is the constant provided in \eqref{alpha invariant}. Then, utilizing inequality \eqref{alpha invariant}, we observe that
\[\int_{B_s} \exp\Big\{\beta \frac{(-\varphi_s)^{\frac{2m+1}{2m}}}{\Phi_{s,k}^{\frac{1}{2m}}}\Big\}{\rm d}\sigma\leq \int_{B} \exp\{\beta c(m)(-\varphi_{s,k}) \} {\rm d}\sigma\leq C_{4}\,.
\]
Here $C_{4}$ is a positive constant independent of $k$. As $k$ approaches infinity, this concludes the proof of the proposition.
\end{proof}

Having previously established that $\Psi$ can be bounded in terms of $\Phi$ in Lemma \ref{Phi A}, we now shift our focus to the converse scenario. Specifically, we present the following
\begin{lemma}
Given $p>2m$, let $\delta_0=\frac{p-2m}{2pm}>0$. Then, there exist a constant $C_{5}$ that depends only on $m$, $p$, $\omega$, $\omega_{h}$, $\|e^{2G}\|_{L^1(\log L)^p}$ and $r_0$, such that 
\begin{equation}\label{A Phi}
\Phi_{s}\leq C_{5} \Psi_{s}^{1+\delta_0}\,.
\end{equation}
\end{lemma}
\begin{proof}
Firstly, we assert that for every   $p>2m$, there exists a constant $C_{6} $ that depends only on $m$, $p$, $\beta,$ $\|e^{2G}\|_{L^1(\log L)^p}$ and $r_0$, satisfying the following inequality:
\begin{equation}\label{phi to psi}
\int_{B_s}(-\varphi_s)^{\frac{p(2m+1)}{2m}}e^{2G}{\rm d}\sigma\leq C_{6}\Phi_{s}^{\frac{p}{2m}}\,.
\end{equation}

Given this assertion, one can subsequently apply the H\"older inequality to derive
\begin{equation*}
\begin{split}
\Phi_{s} =& \int_{B_s} (-\varphi_s) e^{2G}{\rm d}\sigma \\
\leq &  \Big(\int_{B_s} (-\varphi_s)^{\frac{p(2m+1)}{2m}} e^{2G}{\rm d}\sigma\Big )^{\frac{2m}{p(2m+1)}} \Big(\int_{B_s} e^{2G}{\rm d}\sigma \Big)^{1-\frac{2m}{p(2m+1)}}\\
\leq & C_{6}^{\frac{m}{p(2m+1)}} \Phi_{s}^{\frac{1}{2m+1}} \Psi_{s}^{1-\frac{2m}{p(2m+1)}}\,.
\end{split}
\end{equation*}
As a result, we obtain the desired inequality \eqref{A Phi} with $C_{5}=C_{6}^{\frac{1}{p}}$.

To demonstrate the aforementioned assertion, for the sake of notational simplicity, we denote
$$U_s:=\frac{(-\varphi_s)^{\frac{(2m+1)}{2m}}}{\Phi_{s}^{\frac{1}{2m}}}.$$
Then, the established inequality \eqref{trudinger} can be reformulated as
\begin{equation}\label{trudinger 2}
\int_{B_s} \exp\{\beta U_s\} {\rm d}\sigma \leq C_{4}\,.    
\end{equation}
To establish the targeted inequality \eqref{phi to psi}, it suffices to check the following equivalent inequality
\begin{equation}\label{p form of Us}
\int_{B_s} U_s^{p} e^{2G}{\rm d}\sigma\leq C_{6}\,.
\end{equation}
By adapting the well-known Young's inequality,  Guo-Phong \cite{GP22} demonstrated that that for every non-negative function $v$ on $M$, the following inequality holds true:
\begin{equation}\label{Young}
\frac{1}{2^p} v^p e^{2G}\leq e^{2G}(1+|G|)^p +C_p e^{v}\,,
\end{equation}
where $C_p$ is a positive constant depending only on $p$. Considering the substitution $v=\beta U_s$ in \eqref{Young} and integrating the resulting expression over $B_s$, one can subsequently deduce from \eqref{trudinger 2} that
\[
\begin{split}
&\frac{\beta^p}{2^p} \int_{B_s} U_s^{p} e^{2G} {\rm d}\sigma\\
&\leq \int_{B_s} e^{2G} (1+|G|)^p {\rm d}\sigma +C_p \int_{B_s} \exp\{\beta U_s\} {\rm d}\sigma \\
&\leq \|e^{2G}\|_{L^1(\log L)^p} +C_pC_{4}\,.
\end{split}
\]
Then, the desired inequality \eqref{p form of Us} follows with $C_{6}=\frac{2^p}{\beta^p}(\|e^{2G}\|_{L^1(\log L)^p} +C_pC_{4})$.  This validates the stated assertion, thereby concluding our proof of the lemma.
\end{proof}

Having established the preceding lemma, in conjunction with Lemma \ref{Phi A}, we are able to derive the following result:
\begin{lemma}
Define $\Psi(s):=\Psi_{s}$ as an induced function over the open interval $(0,c_0r_0^2)$. Then, $\Psi: (0,c_0r_0^2)\rightarrow \R_{+}$ is a monotone increasing function, possessing the following  properties
\begin{equation*}
\lim_{s\rightarrow 0^{+}}\Psi(s)=0, \qquad t\Psi(s-t) \leq C_{5} \Psi(s)^{1+\delta_0}.
\end{equation*}
\end{lemma}
\begin{proof}
We adopt a similar line of arguments as employed  in the proof of Lemma 5 in \cite{GP22}. Now, we proceed to verify  
\begin{itemize}
\item[(a)] the monotonicity of $\Psi$;
\item[(b)] the convergence of $\Psi(s)$ to $0$ as $s$ approaches to $0^{+}$. 
\end{itemize}

Firstly, by the definition of $B_s$,  we observe that the relation $B_{t}\subset B_s$ holds true whenever $0<t<s<c_0r_0^2$, which in turn implies (a). 

Secondly, due to the choice of $z_{\mathrm{min}}$, the function $$\varphi(z)-\varphi(z_{\mathrm{min}})+c_0|z-z_{\mathrm{min}}|^2$$ 
is non-negative within the set $B$ and, notably,  it vanishes at the single point $z_{\mathrm{min}}$. This yields that $\underset{s>0}{\cap} \overline{B_s}=\{z_{\mathrm{min}}\}$ and whence (b) is proved.
\end{proof}
	
By adopting a standard iterative approach akin to those employed in \cite{De57} and \cite{Ko98}, B. Guo and D. Phong \cite{GP22} successfully  demonstrate the following
\begin{lemma}\cite[Lemma 6]{GP22}\label{gp}
Suppose  the function $\phi(t)$ is monotonically increasing on the interval $(0,t_{0})$ and fulfills the following properties
\begin{itemize}
\item[(i)] $\lim_{t\rightarrow 0^{+}}\phi(t)=0$;
\item[(ii)] $\phi$ is positive on the whole interval $(0,t_{0})$;
\item[(iii)] for any $ 0<t'<t<t_{0}$, $t'\phi(t-t')\leq N\big[\phi(t)\big]^{1+\delta}$ for some positive constants $N$ and $\delta$.
\end{itemize}
Then we have
\[
\phi(t_{0})\geq \Big(\frac{t_{0}(1-2^{-\delta})}{2N} \Big)^{\frac{1}{\delta}}.
\]
\end{lemma}
A direct implication of  Lemma \ref{gp} is the existence of a uniformly positive lower bound for $\Psi_{c_{0}r_{0}^2}$ as the following
\begin{lemma}\label{lower bound of psi}
There exists a positive constant $C_{2}$ depends only on $c_0$, $r_0$, $C_{5}$ and $\delta_0$, such that
\begin{equation}\label{GP22}
\Psi_{c_0r_0^2}\geq C_{2}\,.
\end{equation}
\end{lemma}
\begin{proof}
The proof of Lemma 7 in \cite{GP22} directly applies to this context and, for brevity, we omit its repetition here.
\end{proof}

\section{$L^q$ estimate and proof of Proposition \ref{lower bound}}\label{proof of main proposition}
In this section, we aim to derive the lower bound estimate $\varphi(z_{\min})$ as exhibited in Proposition \ref{lower bound}. Our proof will crucially depend on not only the  inequality \eqref{GP22} presented in Lemma \ref{lower bound of psi} but also on the $L^q$ estimate  elaborated upon  in the subsequent subsection.
\subsection{The $L^q$ estimate}

\begin{lemma}\label{L q estimate}
Let $\varphi$ be the smooth solution of \eqref{new n-1 MA}. Then, there exist positive constants $q, C$ depend only on the given data, such that
\begin{equation}\label{L q integrability}
\int_{M}(-\varphi)^{q}{\rm d}\sigma\leq C\,.
\end{equation}
\end{lemma}
\begin{proof}
For the proof, we will adopt an approach inspired by \cite{Sze} and \cite{GZ22}, which focuses on deriving $L^q\,$ estimates for a broader class of Monge-Amp\`ere type equations.

Firstly, the quaternionic $(n-1)$-plurisubharmonicity assumption $\varphi\in \PSH_J(M,\omega,\omega_{h})$  inherently implies that
\[
\tr_{\omega}\omega_{h}+\Delta_{I,g}\varphi+\tr_{\omega}W[\varphi]\geq 0\,.
\]
This, in turn, yields a lower bound for the drift Laplacian of $ \varphi $:
\begin{equation}\label{eq_in_C0}
\Delta_{d}\varphi:=\Delta_{I,g}\varphi+\tr_{\omega}W[\varphi] \geq -\sup_{M} \tr_{\omega}\omega_{h}:=-C\,.
\end{equation}
It is evident that $\Delta_{d}$ is an elliptic operator of second order. 

Now, we proceed to verify the desired inequality \eqref{L q integrability} by using the weak Harnack inequality. Select an appropriate open covering of $ M $ comprised of coordinate balls $ \{ B_{2r_i}(x_i)\}_{i=1}^{N}  $ such that a subset of these balls,  namely $ \{ B_i=B_{r_i}(x_i) \}_{i=1}^{N} $, still covers $ M $. Since $ \varphi $ is non-positive and satisfies the inequality \eqref{eq_in_C0}, application of  the weak Harnack inequality (cf. \cite[Theorem 9.22]{GT}) leads to the conclusion that
\[
\|\varphi\|_{L^q(B_i)}=\Big( \int_{B_i}(-\varphi)^q \Big)^{\frac{1}{q}}\leq C\Big( \inf_{B_i}(-\varphi)+1 \Big)\,,
\]
where $q,C>0 $ are constants that depend only on the covering and the background metric.

Observing that we have assumed $ \sup_M \varphi=0 $, it follows that there exists at least one index $ j $ for which $ \inf_{B_j}(-\varphi)=-\sup_{B_j}\varphi=0 $. Hence, we have $ \|\varphi\|_{L^q(B_j)}\leq C $. This bound can be extended to all balls $ B_i $ such that $ B_i\cap B_j\neq \emptyset $, indeed one can bound $ \inf_{B_i}(-\varphi) $ from above in terms of $ \|\varphi\|_{L^q(B_j)}\,$:
\[
\begin{split}
\inf_{B_i} (-\varphi)&\leq \inf_{B_i\cap B_j}(-\varphi) \leq \frac{1}{\mathrm{vol}(B_i\cap B_j)^{\frac{1}{q}}}\|\varphi\|_{L^q(B_i\cap B_j)}\\
&\leq \frac{1}{\mathrm{vol}(B_i\cap B_j)^{\frac{1}{q}}}\|\varphi\|_{L^q(B_j)}\,.    
\end{split}
\]
Reiterating the argument and in a finite number of steps, we will have bound $ \|\varphi\|_{L^q(B_i)} $ for each $ i=1,\cdots,N $, and thus also $ \|\varphi\|_{L^q(M)}. $ This completes our proof of inequality \eqref{L q integrability}.
\end{proof}

\subsection{Proof of Proposition \ref{lower bound}}

Without loss of generality, we can assume that  $\varphi(z_{\mathrm{min}})<-2$.  If necessary, shrink the ball $B$  to ensure that its radius  $r_0\leq \frac{1}{2}.$ Taking into account all these factors, we observe that for every point $z\in B_{c_0r_0^2}\,$, it follows that 
\[-\varphi(z)-c_0|z-z_{\mathrm{min}}|^2> -\varphi(z_{\mathrm{min}})-c_0r_0^{2}> 2-c_0 \cdot (\frac{1}{2})^2> 1.\]
Therefore, the function
\begin{equation}\label{positive condition}
\log \frac{-\varphi(z)-c_0|z-z_{\mathrm{min}}|^2}{(-\varphi(z_{\mathrm{min}})-c_0r_0^2)^{\frac{1}{2}}}\geq \frac{1}{2} \log \Big(-\varphi(z_{\mathrm{min}})-c_0r_0^2 \Big)
\end{equation}
is strictly  positive in $B_{c_0r_0^2}$. 

Select a number $p$ such that $p>2m$, and then we define
\[
v:= p \Big(\log \frac{-\varphi(z)-c_0|z-z_{\mathrm{min}}|^2}{(-\varphi(z_{\mathrm{min}})-c_0r_0^2)^{\frac{1}{2}}}\Big)^{\frac{1}{p}}>0\,.
\]
Given that the inequality $pt^{\frac{1}{p}}\leq  qt+ C_{p,q}$ holds for all $t\geq 0$, with  $C_{p,q}$ being a positive constant $C_{p,q}$ that depends only on both $p$ and $q$, where $q$ is a constant as specified  in Lemma \ref{L q estimate}, we can then apply  subsequently apply this inequality to $v$, and as a consequence,   the inequality \eqref{Young} leads to the following
\begin{equation}\label{Guo-Phong type inequality}
\begin{split}
&\frac{p^p}{2^p}\log \frac{-\varphi(z)-c_0|z-z_{\mathrm{min}}|^2}{(-\varphi(z_{\mathrm{min}})-c_0r_0^2)^{\frac{1}{2}}} e^{2G}\\ &=\frac{1}{2^p} v^p e^{2G} \leq e^{2G}(1+|G|)^p+ C_p e^{v} \\
& \leq e^{2G}(1+|G|)^p +C_p \exp\Big\{q\log \frac{-\varphi(z)-c_0|z-z_{\mathrm{min}}|^2}{(-\varphi(z_{\mathrm{min}})-c_0r_0^2)^{\frac{1}{2}}}+C_{p,q} \Big\}\\
&=e^{2G}(1+|G|)^p+C_p\exp\{C_{p,q}\}\Big(\frac{-\varphi(z)-c_0|z-z_{\mathrm{min}}|^2}{(-\varphi(z_{\mathrm{min}})-c_0r_0^2)^{\frac{1}{2}}}\Big)^{q}\,.
\end{split}   
\end{equation}
By integrating over  $B_{c_0r^2_0}\,$, we conclude that
\begin{equation*}\label{final}
\begin{split}
&\frac{p^p}{2^{p+1}}C_{2}\log (-\varphi(z_{\mathrm{min}})-c_0r^2_0)  \\ 
&\leq \int_{B_{c_0r^2_0}} \frac{p^p}{2^p}\log \frac{-\varphi(z)-c_0|z-z_{\mathrm{min}}|^2}{(-\varphi(z_{\mathrm{min}})-c_0r_0^2)^{\frac{1}{2}}} e^{2G}{\rm d}\sigma\\
& \leq \|e^{2G}\|_{L^1(\log L)^p}+\frac{C_p\exp\{C_{p,q}\}}{\Big(-\varphi(z_{\mathrm{min}})-c_0r_0^{2}\Big)^{\frac{q}{2}}} \int_{B_{c_0r^2_{0}}} \Big(-\varphi(z)-c_0|z-z_{\mathrm{min}}|^2\Big)^{q}{\rm d}\sigma\\
& \leq \|e^{2G}\|_{L^1(\log L)^p}+\frac{C_{7}C_p\exp\{C_{p,q}\}}{(-\varphi(z_{\mathrm{min}})-c_0r_0^{2})^{\frac{q}{2}}}\,, 
\end{split}
\end{equation*}
which yields the desired inequality stated in this proposition. Here, the second line utilizes inequalities \eqref{GP22} and \eqref{positive condition}, while the third line stems from the application of \eqref{Guo-Phong type inequality}. The fourth line, in turn, is a direct consequence of \eqref{L q integrability}. This concludes our proof of the proposition.
\qed

\end{sloppypar}	
\end{document}